\documentclass[12pt]{article}
\usepackage[utf8]{inputenc}
\usepackage[english]{babel}
\usepackage{amsfonts}
\usepackage{amsmath, amssymb}
\usepackage{amsthm}
\usepackage{amscd}
\usepackage{cite}
\usepackage[left=2cm,right=2cm, top=2cm,bottom=2cm,bindingoffset=0cm]{geometry}

\newtheorem{theorem}{Theorem}
\newtheorem{lemma}{Lemma}
\newtheorem{prop}{Proposition}

\theoremstyle{definition}

\theoremstyle{remark}
\newtheorem*{remark}{Remark}
\theoremstyle{remark}

\title{Equivariant cyclic cocycles on the Boutet de Monvel symbol algebra}
\author{A.V. Boltachev, A.Yu. Savin}

\begin{document}

\maketitle

\begin{abstract}
We construct periodic cyclic cocycles on the algebra of symbols of Boutet de Monvel operators and use them to interpret   
the index formula for elliptic pseudodifferential boundary value problems due to Fedosov  as the Chern--Connes pairing of 
the classes in $K$-theory of elliptic symbols with cyclic cocycles. We also consider the equivariant case. 
Namely, we construct periodic cyclic cocycles on the crossed product of the algebra of symbols with a group acting 
on this algebra by automorphisms. Such crossed products arize in index theory of nonlocal boundary 
value problems with shift operators.
\end{abstract}

\section{Cyclic cocycles on the Boutet de Monvel symbol algebra}
\label{main_results}

\paragraph{Boutet de Monvel symbol algebra.}
Consider $M=\mathbb{R}^2_+$ with the coordinates $(x_1,x_2)$ and the boundary $X=\mathbb{R}$ defined by $x_2=0$.
 Consider the algebra of symbols of Boutet de Monvel operators on $M$ (see~\cite{Bout2,ReSc1}). 
The symbols are assumed to have compact supports. Recall that the elements of algebra $\mathcal{A}$ are the pairs
\begin{equation}
\label{BdM_sym_def}
\widetilde{a}=(a,a_X)\in C^\infty_c(S^*M)\oplus C^\infty_c(S^*X,\mathcal{B}(H_+\oplus\mathbb{C})),
\end{equation}
where $a$ is a smooth function on the cosphere bundle $S^*M$ called the {\em{interior symbol}} satisfying the transmission property:
for all $k\in\mathbb{Z}_+$ and arbitrary $\alpha\in\mathbb{Z}_+$ the following equality holds:
\begin{equation}
\label{trans_prop}
D^k_{x_2}D^\alpha_{\xi_1}a(x_1,0,0,\xi_2)=(-1)^{\alpha}D^k_{x_2}D^\alpha_{\xi_1}a(x_1,0,0,-\xi_2),\quad \xi_2\ne 0,
\end{equation}
where
$$
D^\alpha_{\xi_1}=\left(-i\frac{\partial }{\partial \xi_1}\right)^{\alpha}.
$$
The function $a_X$ in~\eqref{BdM_sym_def} called the {\em{boundary symbol}} is a smooth operator function on 
$$
  S^*X=\{(x_1,\xi_1)\big| x_1\in X, |\xi_1|=1\}=\mathbb{R}\times\{\pm1\}
$$ 
of the form
\begin{equation}
\label{eq-spb1}
a_X=
\begin{pmatrix}
\Pi_+a|_{\partial T^*M} + \Pi'g & c \\[2mm]
\Pi'_{\xi_2}b & r
\end{pmatrix}
\colon
\begin{array}{c}
   H_+\\ 
   \oplus \\
   \mathbb{C}
 \end{array}
 \longrightarrow
 \begin{array}{c}
   H_+\\ 
   \oplus \\
   \mathbb{C}
 \end{array}
\end{equation}
Here
\begin{itemize}
\item $H_+=\mathcal{F}_{x_2\to \xi_2}(\mathcal{S}(\overline{\mathbb{R}}_+))$ is the space of Fourier images of the Schwartz space $\mathcal{S}(\overline{\mathbb{R}}_+)$ of smooth rapidly decreasing at infinity functions on $\overline{\mathbb{R}}_+$;
\item $a\in C^\infty_c(S^*M)$; $r\in C^\infty_c(S^*X)$;
\item $b\in C^\infty_c(T^*X,H_-)$; $c\in C^\infty_c(S^*X, H_+)$;
\item $g\in C^\infty(T^*X,H_+\otimes H_-)$, where $H_-=\mathcal{F}(\mathcal{S}(\overline{\mathbb{R}}_-))$;
\item $\Pi_\pm:H_+\oplus H_-\to H_\pm$ are projections, while $\Pi'$ is the following functional 
$$
\begin{array}{ccc}
\Pi':H_+\oplus H_- & \longrightarrow & \mathbb{C},\\[2mm]
u(\xi_2)& \longmapsto & \displaystyle \lim\limits_{x_2\to 0+} \mathcal{F}^{-1}_{\xi_2\to x_2}(u(\xi_2)).
\end{array}
$$
\end{itemize}
 
Denote by $\mathcal{A}$ the unitization of the algebra of symbols of Boutet de Monvel operators described above.

\paragraph{Periodic cyclic cohomology.}
Let us recall the definition of the periodic cyclic cohomology in terms of operators $(b,B)$ (see~\cite{Con1}). Denote by $C^n(\mathcal{A})={\rm{Hom}}(\mathcal{A}^{n+1},\mathbb{C})$ the space of $(n+1)$-linear functionals $\varphi(a_0,...,a_n)$, where $a_0,...,a_n\in\mathcal{A}$. Such functionals are called   $n$-{\em{cochains}}.
We consider the following operators:
\begin{itemize}
\item 
$B_0:C^n(\mathcal{A})\longrightarrow C^{n-1}(\mathcal{A}),$
$$
(B_0\varphi)(a_0,...,a_{n-1})=\varphi(1,a_0,...,a_{n-1}) - (-1)^n\varphi(a_0,...,a_{n-1},1);
$$
\item $N:C^n(\mathcal{A})\longrightarrow C^n(\mathcal{A})$
$$
N\varphi(a_0,...,a_{n-1})=\sum\limits_{j=0}^{n-1}{(-1)^{(n-1)j}\varphi(a_j,...,a_{j-1})};
$$
\item $B:C^n(\mathcal{A})\longrightarrow C^{n-1}(\mathcal{A}),
B=NB_0.$
\item the Hochschild differential $b:C^n(\mathcal{A})\longrightarrow C^{n+1}(\mathcal{A})$
$$
(b\varphi)(a_0,...,a_{n+1})=\sum\limits^n_{i=0}{(-1)^i\varphi(a_0,...,a_ia_{i+1},...,a_{n+1})} + (-1)^{n+1}\varphi(a_{n+1}a_0,...,a_n).
$$
\end{itemize}

From the relations $b^2=0, B^2=0, bB+Bb=0$, it follows that the following periodic complex is defined
\begin{equation}
\label{period_comp}
...\stackrel{B+b}{\longrightarrow}
C^{ev}(\mathcal{A})\stackrel{B+b}{\longrightarrow}
C^{odd}(\mathcal{A})\stackrel{B+b}{\longrightarrow}
C^{ev}(\mathcal{A})\stackrel{B+b}{\longrightarrow}...,
\end{equation}
where $C^{ev}(\mathcal{A})=\bigoplus\limits_k C^{2k}(\mathcal{A})$, $C^{odd}(\mathcal{A})=\bigoplus\limits_k C^{2k+1}(\mathcal{A})$.
{\it The periodic cyclic cohomology} of   $\mathcal{A}$ is the cohomology of the complex~\eqref{period_comp} denoted by
$$
HP^*(\mathcal{A})=H^*(C^{ev/odd}(\mathcal{A}),B+b).
$$

\paragraph{Periodic cyclic cocycles.}
We define the pair of cochains $(\varphi_1,\varphi_3)\in C^{odd}(\mathcal{A})$
\begin{equation}
\label{odd_cochains}
\varphi_3(a_0,a_1,a_2,a_3)=\int\limits_{S^*M}{a_0da_1da_2da_3},\qquad 
\varphi_1(a_0,a_1)=\frac{1}{2}\int\limits_{S^*X}{{\rm{tr}}'(a_0da_1-a_1da_0)}.
\end{equation}
Here in the first formula $d$ is the exterior differential on $S^*M$ and we integrate the expression defined by the interior symbols. In the second formula in~\eqref{odd_cochains}, $d$ is the exterior differential on $S^*X\simeq\mathbb{R}^1\times\{\pm1\}$ and ${\rm{tr}'}$ is the regularized trace from~\cite{Fds17}
$$
{\rm{tr}'}
\begin{pmatrix}
\Pi_+a + \Pi'g & c \\[2mm]
\Pi'b & r
\end{pmatrix}
=
\Pi'_{\xi_2}g(\xi_2,\xi_2) + r.
$$

\begin{theorem}
\label{cohom_class}
The pair $(\varphi_1, \frac{i}{4\pi}\varphi_3)$ defines a class in the periodic cyclic cohomology $HP^{odd}(\mathcal{A})$. 
This means that the following equalities are valid
\begin{equation}
\label{bB_rels}
B\varphi_1=0;\qquad
-\frac{i}{4\pi}B\varphi_3=b\varphi_1;\qquad
b\varphi_3=0.
\end{equation}
\end{theorem}
\begin{proof}
1. Let us first prove the left most equality in~\eqref{bB_rels}:
$$
B\varphi_1(a_0)=\varphi_1(1,a_0)-\varphi_1(a_0,1)=\frac{1}{2}\int\limits_{S^*X}{{\rm{tr}}'(a_{0x_1})dx_1}+\frac{1}{2}\int\limits_{S^*X}{{\rm{tr}}'(a_{0x_1})dx_1}=0,
$$
where $a_{0x_1}= \partial a_0/\partial x_1$ and we used the fact that integrals of exact forms over $S^*X$ are equal to zero.

2. Let us now prove that $b\varphi_3=0$. By the definition of the Hochschild differential $b$ we have
\begin{multline*}
(b\varphi_3)(a_0,a_1,a_2,a_3,a_4)=\\
\varphi_3(a_0a_1,a_2,a_3,a_4) - 
\varphi_3(a_0,a_1a_2,a_3,a_4) + \varphi_3(a_0,a_1,a_2a_3,a_4) - 
\varphi_3(a_0,a_1,a_2,a_3a_4) + 
\varphi_3(a_4a_0,a_1,a_2,a_3)=
\\
=\int{\left(a_0a_1\,da_2da_3da_4 - 
a_0\,d(a_1a_2)da_3da_4 + 
a_0\,da_1d(a_2a_3)da_4 - 
a_0\,da_1da_2d(a_3a_4) + 
a_4a_0\,da_1da_2da_3\right)}=
\\
=\int{\left(a_0a_1\,da_2da_3da_4 - 
a_0\,(da_1)a_2da_3da_4 - 
a_0a_1\,da_2da_3da_4 + 
a_0\,da_1(da_2)a_3da_4\right.} +
\\
\int{\left.a_0\,(da_1)a_2da_3da_4 - 
a_0\,da_1da_2(da_3)a_4 -
a_0\,da_1(da_2)a_3da_4 + 
a_4a_0\,da_1da_2da_3\right)}=
\\
=
\int{\left(a_0\,da_1da_2(da_3)a_4 - 
a_4a_0\,da_1da_2da_3\right)}=0.
\end{multline*}

3. Finally, we prove the second equality in~\eqref{bB_rels}. To this end we calculate $b\varphi_1$:
\begin{multline*}
2b\varphi_1(a_0,a_1,a_2)=2\left(\varphi_1(a_0a_1,a_2) - \varphi_1(a_0,a_1a_2) + \varphi_1(a_2a_0,a_1)\right)=
\\
\int\limits_{S^*X}{{\rm{tr}}'(a_0a_1da_2-a_2d(a_0a_1) - 
\int\limits_{S^*X}{\rm{tr}}'a_0d(a_1a_2)-a_1a_2da_0+
\int\limits_{S^*X}{\rm{tr}}'a_2a_0da_1-a_1d(a_2a_0))}=
\\
\int\limits_{S^*X}{{\rm{tr}}'(
a_0a_1da_2-a_2d(a_0a_1)-a_0da_1a_2-a_0a_1da_2+a_1a_2da_0+a_2a_0da_1-a_1da_2a_0-a_1a_2da_0
)}=
\\
\int\limits_{S^*X}{{\rm{tr}}'(
-a_2d(a_0a_1)-a_0da_1a_2+a_2a_0da_1-a_1da_2a_0
)}=
\\
\int\limits_{S^*X}{{\rm{tr}}'(
da_2a_0a_1-a_0da_1a_2+a_2a_0da_1-a_1da_2a_0
)}=
\int\limits_{S^*X}{\left({\rm{tr}}'[a_2,a_0da_1]-
{\rm{tr}}'[a_1,da_2a_0]\right)}.
\end{multline*}
Here in the fifth equality we integrated by parts the first term.
Now let us use the formula from~\cite{Fds17} for the regularized trace of commutators. We obtain
\begin{multline}
\label{b_commutators}
2b\varphi_1(a_0,a_1,a_2)=-\int\limits_{S^*X}{\left(i\Pi'\left(a_{2\xi_2}a_0da_1\right)
- i\Pi'\left(a_{1\xi_2}da_2a_0\right)\right)}=
\\
-\frac{i}{2\pi}\int\limits_{\mathbb{S}^1_{x_1}\times\mathbb{R}_{\xi_2}}{\hspace{-5pt}\left.\left(a_0da_1a_{2\xi_2} - a_0a_{1\xi_2}da_2\right)\right|^{\xi_1=1}_{\xi_1=-1}d\xi_2}=
-\frac{i}{2\pi}\int\limits_{\mathbb{S}^1_{x_1}\times\mathbb{R}_{\xi_2}}{\hspace{-5pt}\left.\left(a_0a_{1x_1}a_{2\xi_2} - a_0a_{1\xi_2}a_{2x_1}\right)\right|^{\xi_1=1}_{\xi_1=-1}dx_1d\xi_2}=
\\
-\frac{i}{2\pi}\int\limits_{\mathbb{S}^1_{x_1}\times\mathbb{R}_{\xi_2}}{\hspace{-5pt}\left.a_0\left(a_{1x_1}a_{2\xi_2} - a_0a_{1\xi_2}a_{2x_1}\right)\right|^{\xi_1=1}_{\xi_1=-1}dx_1d\xi_2}.
\end{multline}
Here in the second equality we replaced the functional $\Pi'$ by the integral, since $(\dots)|_{\xi_1=-1}^{\xi_1=1}$ is of order $O(\xi_2^{-2})$ at infinity and the integral converges.

Now let us calculate $B\varphi_3$.
For operators $B_0$ and $B$ we have
\begin{equation*}
(B_0\varphi_3)(a_0,a_1,a_2) = \varphi_3(1,a_0,a_1,a_2)
+\varphi_3(a_0,a_1,a_2,1)=\varphi_3(1,a_0,a_1,a_2).
\end{equation*}
\begin{multline}
\label{B_int1}
(B\varphi_3)(a_0,a_1,a_2) = (NB_0\varphi_3)(a_0,a_1,a_2) = \varphi_3(1,a_0,a_1,a_2) + \varphi_3(a_0,a_1,a_2, 1) \\
+ \varphi_3(a_1,a_2,1,a_0) + \varphi_3(a_2,1,a_0,a_1) = \varphi_3(1,a_0,a_1,a_2)=\\
=\int\limits_{S^*M}{da_0da_1da_2}=\int\limits_{\partial(S^*M)}{\hspace{-5pt}a_0da_1da_2}.
\end{multline}
Let us use the following change of variables in the last integral in~\eqref{B_int1}: 
\begin{equation}
\label{var_change}
\begin{array}{ccc}
\mathbb{S}^1\times\left(\mathbb{R}_{\xi_1=1}\cup\mathbb{R}_{\xi_1=-1}\right) & \longrightarrow & \partial(S^*M)=\mathbb{S}^1_{x_1}\times\mathbb{S}^1_{\xi_1,\xi_2}\vspace{2mm}\\
(x_1,\pm1,\xi_2) & \stackrel{f}{\longmapsto} & \displaystyle\left(x_1,\frac{\pm1}{\sqrt{\xi_2^2+1}}, \frac{\xi_2}{\sqrt{\xi_2^2+1}}\right).
\end{array}
\end{equation}
Then we obtain
\begin{multline}
\label{B_phi_3_res}
(B\varphi_3)(a_0,a_1,a_2) =\int\limits_{\partial(S^*M)}{a_0da_1da_2}=\int\limits_{\mathbb{S}^1\times \mathbb{S}^1}{a_0da_1da_2}=\int\limits_{\mathbb{S}^1\times\left(\mathbb{R}_{\xi_1=1}\cup\mathbb{R}_{\xi_1=-1}\right)}{\hspace{-30pt}f^*(a_0da_1da_2)}=
\\
\int\limits_{\mathbb{S}^1\times \mathbb{R}}{a_0da_1da_2\big|_{\xi_1=-1}^{\xi_1=1}}=
\int\limits_{\mathbb{S}^1\times \mathbb{R}}{a_0\left.\left(a_{1x_1}dx_1 + a_{1\xi_2}d\xi_2\right)\left(a_{2x_1}dx_1 + a_{2\xi_2}d\xi_2\right)\right|_{\xi_1=-1}^{\xi_1=1}}=
\\
\int\limits_{\mathbb{S}^1\times \mathbb{R}}{a_0\left.\left(a_{1x_1}a_{2\xi_2} - a_{1\xi_2}a_{2x_1}\right)\right|_{\xi_1=-1}^{\xi_1=1}dx_1\wedge d\xi_2}.
\end{multline}
Here in the fourth integral we used the homogeneity of symbols.

Finally, comparing~\eqref{B_phi_3_res} with~\eqref{b_commutators}, we obtain the desired relation
$$
b\varphi_1=-\frac{i}{4\pi}B\varphi_3.
$$
\end{proof}

\section{Equivariant cyclic cocycles}

As in Section~\ref{main_results}, let $\mathcal{A}$ be the algebra of Boutet de Monvel symbols over $M$. 
Let us consider a group $\Gamma\subset{\rm{Diff}}(M)$ of diffeomorphisms such that $\Gamma(X)=X$. 
We denote symbols in $\mathcal{A}$ by $\widetilde{a}=(a,a_X)$ and consider the group action on these elements.

$\Gamma$ acts on $\mathbb{R}^2_+$ by diffeomorphisms $\gamma:\mathbb{R}^2_+\longrightarrow\mathbb{R}^2_+$. 
The action of $\Gamma$ on the interior symbols is given by
$$
a\longmapsto {(\partial\gamma)^{*}}^{-1}a,
$$
where $\partial\gamma:T^*M\longrightarrow T^*M$ is the codifferential.
In its turn, the action of $\Gamma$ on the boundary symbols is given by
$$
\begin{array}{rcl}
C^\infty(S^*X,\mathcal{B}(H_+\oplus\mathbb{C}))& \longrightarrow & C^\infty(S^*X,\mathcal{B}(H_+\oplus\mathbb{C}))
\\[2mm]
a_X(x,\xi)& \longmapsto &\varkappa_\lambda a_X(\partial\gamma^{-1}(x,\xi)) \varkappa^{-1}_\lambda,\qquad \lambda=\lambda(x_1)=\left.\displaystyle\frac{\partial\gamma_2}{\partial x_2}\right|_{x_2=0},
\end{array}
$$
where
$$
\begin{array}{rcl}
\varkappa_\lambda:H_+\oplus\mathbb{C} & \longrightarrow & H_+\oplus\mathbb{C}\\[2mm]
(u(\xi_2),v) & \longmapsto & (\lambda^{1/2} u(\lambda\xi_2),v).
\end{array}
$$
Here the action of $\Gamma$ on $M$ and $X$ is lifted to the bundles ${T}^*M$ and $T^*X$ using codifferentials $\partial \gamma=(d\gamma^t)^{-1}$
of the corresponding diffemorphisms (here $d\gamma$ is the differential of $\gamma$, while $d\gamma^t$ is its dual mapping of the cotangent bundle).
For brevity we denote the group actions given above by
$$
(a,a_X)\longmapsto \gamma(a,a_X).
$$
The action of $\Gamma$ on $\mathcal{A}$ allows to define the algebraic crossed product denoted by $\mathcal{A}\rtimes\Gamma$. Its elements are  functions on $\Gamma$ with values in $\mathcal{A}$ denoted by $\widetilde{a}=\{\widetilde{a}(\gamma)\}$.
Let us recall the definition of the product of elements in $\mathcal{A}\rtimes\Gamma$. Given $\{\widetilde{a}(\gamma)\},\{\widetilde{b}(\gamma)\}\in\mathcal{A}\rtimes\Gamma$, the product is defined by
$$
\{\widetilde{a}(\gamma)\}\cdot\{\widetilde{b}(\gamma)\}=\left\{\sum\limits_{\gamma_1\gamma_2=\gamma}\widetilde{a}(\gamma_1)\gamma_1(\widetilde{b}(\gamma_2))\right\}.
$$
Let us now give an analogue of the formula for the regularized trace of commutator~\cite{Fds17}. Denote by $\Omega_X$ the differential graded algebra generated by the boundary symbols and the differential forms on $S^*X$.
\begin{prop}
\label{trace_prop}
Given $\widetilde{\omega}_1,\widetilde{\omega}_2\in\Omega_X\rtimes\Gamma$, we have
\begin{equation}
\label{trace_Fed_gen}
\int\limits_{S^*X}{\rm{tr}'}[\widetilde{\omega}_1,\widetilde{\omega}_2] = -i\int\limits_{S^*X}\Pi'\left(d\omega_1 \omega_2\right),
\end{equation}
where $[\widetilde{\omega}_1,\widetilde{\omega}_2] = \widetilde{\omega}_1\widetilde{\omega}_2 - (-1)^{{\rm{deg}}\widetilde{\omega}_1{\rm{deg}}\widetilde{\omega}_2}\widetilde{\omega}_2\widetilde{\omega}_1$, while $\omega_1, \omega_2$ are the principal symbols of the boundary symbols $\widetilde{\omega}_1,\widetilde{\omega}_2$.
\end{prop}
\begin{proof}
It suffices to consider the forms of the form $\widetilde{\omega}_1=a_1 \delta_\gamma,\widetilde{\omega}_2=a_2 \delta_{\gamma^{-1}}$, where $\delta_\gamma$ is the delta-function on $\Gamma$ with support at $\gamma\in\Gamma$. 

1. Let us first consider the case where $\lambda=1$.
We have
\begin{multline}
\label{Fed_gen_proof_1}
\int\limits_{S^*X}{\rm{tr}'}[\widetilde{\omega}_1,\widetilde{\omega}_2]=
\int\limits_{S^*X}{\rm{tr}'}\left(a_1{\partial\gamma^*}^{-1}(a_2) -(-1)^{{\rm{deg}}a_1{\rm{deg}}a_2}a_2\partial\gamma^*(a_1)\right)=
\\
=\int\limits_{S^*X}\partial\gamma^*\left({\rm{tr}'}\left(a_1{\partial\gamma^*}^{-1}(a_2)\right)\right) - 
(-1)^{{\rm{deg}}a_1{\rm{deg}}a_2}\int\limits_{S^*X}{\rm{tr}'}\left(a_2\partial\gamma^*(a_1)\right)=
\\
\int\limits_{S^*X}{\rm{tr}'}\left(\partial\gamma^*(a_1)a_2-(-1)^{{\rm{deg}}a_1{\rm{deg}}a_2}a_2\partial\gamma^*(a_1)\right)=
\int\limits_{S^*X}{\rm{tr}'}[\partial\gamma^*a_1,a_2]=
-i\int\limits_{S^*X}\Pi'\frac{\partial}{\partial\xi_n}\left(\partial\gamma^*a_1\right)a_2,
\end{multline}
where we used Fedosov's formula. On the other hand,  the right hand side in~\eqref{trace_Fed_gen} is equal to
\begin{equation}
\label{Fed_gen_proof_2}
\int\limits_{S^*X}\Pi'\left(d\omega_1 \omega_2\right) = 
\int\limits_{S^*X}\Pi'\frac{\partial a_1}{\partial \xi_n} {\partial\gamma^*}^{-1}(a_2) =
\int\limits_{S^*X}\Pi'\frac{\partial}{\partial\xi_n}\left(\partial\gamma^*a_1\right)a_2.
\end{equation}
Thus, by \eqref{Fed_gen_proof_1} and~\eqref{Fed_gen_proof_2}, \eqref{trace_Fed_gen} is true.

2. Let us now consider the case where $\partial\gamma={\rm{Id}}$. The left hand side in~\eqref{trace_Fed_gen} is equal to
\begin{equation}
\label{trace_Fed_gen_2}
\tau_X[\widetilde{\omega}_1,\widetilde{\omega}_2]=
\int\limits_{S^*X}{\rm{tr}'}\left(a_1\varkappa_\lambda^{-1}a_2\varkappa_\lambda - a_2\varkappa_\lambda a_1 \varkappa_\lambda^{-1}\right).
\end{equation}
To prove this part of the statement, we use the auxiliary lemma. 
\begin{lemma}
\label{lem_inv}
The following equality is valid
$$
{\rm{tr}'}(\varkappa_\lambda a\varkappa_\lambda^{-1})={\rm{tr}'}a.
$$
\end{lemma}
\begin{proof}
Note that when we find the regularized trace ${\rm{tr}'}$ of a boundary symbol, it suffices to compare only the traces of Green operators in the top left blocks of the matrices. We have
$$
{\rm{tr}'}(\varkappa_\lambda a\varkappa_\lambda^{-1})=\lambda{\rm{tr}'}\left(\Pi'_\eta g(\lambda\xi,\lambda\eta)\right)=\lambda\Pi'g(\lambda\eta,\lambda\eta)={\rm{tr}'}a.
$$
\end{proof}
Let us now move back to the proof of the case 2 of Proposition~\ref{trace_prop}. Applying Lemma~\ref{lem_inv} to the first term in~\eqref{trace_Fed_gen_2}, we obtain
\begin{multline*}
\int\limits_{S^*X}{\rm{tr}'}\left(a_1\varkappa_\lambda^{-1}a_2\varkappa_\lambda - a_2\varkappa_\lambda a_1 \varkappa_\lambda^{-1}\right)=
\int\limits_{S^*X}{\rm{tr}'}\left(\varkappa_\lambda a_1\varkappa_\lambda^{-1}a_2 - a_2\varkappa_\lambda a_1 \varkappa_\lambda^{-1}\right)=
\\
\int\limits_{S^*X}{\rm{tr}'}\left[\varkappa_\lambda a_1\varkappa_\lambda^{-1},a_2\right]=\int\limits_{S^*X}\Pi'\left(\frac{\partial}{\partial\xi_n}(\varkappa_\lambda a_1 \varkappa_\lambda^{-1})a_2\right)=
\int\limits_{S^*X}\int \frac{d}{d\xi}a_1(\lambda\xi)a_2(\xi)d\xi=
\\
\int\limits_{S^*X}\int a'_1(\lambda\xi)a_2(\xi)\lambda d\xi=
\int\limits_{S^*X}\int a'_1(\eta)a_2(\eta/\lambda)d\eta=
\int\limits_{S^*X}\Pi'(d\omega_1\omega_2).
\end{multline*}
Thus we obtained the expression for the right part of~\eqref{trace_Fed_gen}, from which follows that Proposition~\ref{trace_prop} is valid also for the case 2.

\end{proof}

Proposition~\ref{trace_prop} enables us to prove Theorem~\ref{cohom_class} in the case of equivariant symbols as elements in $\mathcal{A}\rtimes\Gamma$.

Let us define the pair of cochains $(\varphi_1,\varphi_3)$:
$$
\varphi_3(a_0,a_1,a_2,a_3)=\int\limits_{S^*M}{(a_0da_1da_2da_3)(e)},\qquad
\varphi_1(a_0,a_1)=\frac{1}{2}\int\limits_{S^*X}{{\rm{tr}}'(a_0da_1-a_1da_0)(e)},
$$
where $a(e)$ stands for the element of the crossed product at $\gamma=e$.

\begin{theorem}
\label{cohom_class_group}
The pair of cochains $\left(\varphi_1, \frac{i}{4\pi}\varphi_3\right)$ defines a class in periodic cyclic cohomology $HP^{odd}(\mathcal{A}\rtimes\Gamma)$. 
This means that the following equalities are valid
\begin{equation}
\label{bB_rels_group}
B\varphi_1=0;\qquad
-\frac{i}{4\pi}B\varphi_3=b\varphi_1;\qquad
b\varphi_3=0.
\end{equation}
\end{theorem}

\begin{proof}
The proof of the first and third equalities in~\eqref{bB_rels_group} is similar to that in the proof of Theorem~\ref{cohom_class}.

Let us prove the second equality in~\eqref{cohom_class_group}.
We have
$$
2b\varphi_1(a_0,a_1,a_2)=2(\varphi_1(a_0a_1,a_2) - \varphi_1(a_0,a_1a_2) + \varphi_1(a_2a_0,a_1))=
$$
$$
\int\limits_{S^*X}{{\rm{tr}}'(
a_0a_1da_2-a_2d(a_0a_1)-a_0d(a_1)a_2-a_0a_1da_2+a_1a_2da_0+a_2a_0da_1-a_1d(a_2)a_0-a_1a_2da_0
)(e)}=
$$
$$
\int\limits_{S^*X}{{\rm{tr}}'(
-a_2d(a_0a_1)-a_0d(a_1)a_2+a_2a_0da_1-a_1d(a_2)a_0
)(e)}=
$$
$$
\int\limits_{S^*X}{{\rm{tr}}'(
(da_2)a_0a_1-a_0d(a_1)a_2+a_2a_0da_1-a_1d(a_2)a_0
)(e)}=
\int\limits_{S^*X}{{\rm{tr}}'[a_2,a_0da_1](e)}-
\int\limits_{S^*X}{{\rm{tr}}'[a_1,d(a_2)a_0](e)}.
$$
Now we use Proposition~\ref{trace_prop} and obtain
\begin{multline}
\label{cc_group_right}
-\int\limits_{S^*X}{i\Pi'\left(\frac{\partial a_2}{\partial \xi_2}a_0da_1\right)(e)}
+\int\limits_{S^*X}{i\Pi'\left(\frac{\partial a_1}{\partial \xi_2}(da_2)a_0\right)(e)}=
\\
-i\frac{1}{2\pi}\int\limits_{\mathbb{S}^1_x}{\int\limits_{\mathbb{R}_{\xi_2}}\left.\left(a_0da_1\frac{\partial a_2}{\partial \xi_2}- a_0\frac{\partial a_1}{\partial \xi_2}(da_2)\right)(e)\right|^{\xi_1=1}_{\xi_1=-1}d\xi_2}
=
\\
-\frac{i}{2\pi}\int\limits_{\mathbb{S}^1_x}{\int\limits_{\mathbb{R}_{\xi_2}}\left.\left(a_0\frac{\partial a_1}{\partial x_1}\frac{\partial a_2}{\partial \xi_2} - a_0\frac{\partial a_1}{\partial \xi_2}\frac{\partial a_2}{\partial x_1}\right)(e)\right|^{\xi_1=1}_{\xi_1=-1}dx_1d\xi_2}
=
\\
-\frac{i}{2\pi}\int\limits_{\mathbb{S}^1_x\times\mathbb{R}_{\xi_2}}{\left.a_0\left(\frac{\partial a_1}{\partial x_1}\frac{\partial a_2}{\partial \xi_2} - a_0\frac{\partial a_1}{\partial \xi_2}\frac{\partial a_2}{\partial x_1}\right)(e)\right|^{\xi_1=1}_{\xi_1=-1}dx_1d\xi_2}.
\end{multline}
On the other hand, consider the left hand side in~\eqref{bB_rels_group}. We have
\begin{equation}
\label{B_int1_group}
(B\varphi_3)(a_0,a_1,a_2) = \int\limits_{\partial(S^*M)}{(a_0da_1da_2)(e)}.
\end{equation}
We use the change of variables~\eqref{var_change} in the last integral in~\eqref{B_int1_group}: 
\begin{multline}
\label{cc_group_left}
(B\varphi_3)(a_0,a_1,a_2)=\int\limits_{\mathbb{S}^1\times \mathbb{S}^1}{(a_0da_1da_2)(e)}=\int\limits_{\mathbb{S}^1\times \mathbb{R}}{\left(a_0da_1da_2\big|_{\xi_1=-1}^{\xi_1=1}\right)(e)}=
\\
\int\limits_{S^1\times \mathbb{R}}{\left[a_0\left.\left(\frac{\partial a_1}{\partial{x_1}}dx_1 + \frac{\partial a_1}{\partial{\xi_2}}d\xi_2\right)\left(\frac{\partial a_2}{\partial{x_1}}dx_1 + \frac{\partial a_2}{\partial{\xi_2}}d\xi_2\right)\right|_{\xi_1=-1}^{\xi_1=1}\right](e)}=
\\
\int\limits_{S^1\times \mathbb{R}}{\left[a_0\left.\left(\frac{\partial a_1}{\partial{x_1}}\frac{\partial a_2}{\partial{\xi_2}} - \frac{\partial a_1}{\partial{\xi_2}}\frac{\partial a_2}{\partial{x_1}}\right)\right|_{\xi_1=-1}^{\xi_1=1}\right](e)dx_1\wedge d\xi_2}.
\end{multline}

Finally, comparing~\eqref{cc_group_right} with~\eqref{cc_group_left}, we obtain the desired equality
$$
b\varphi_1=-\frac{i}{4\pi}B\varphi_3.
$$
\end{proof}

\begin{remark}
Pairings of cocycles in Theorems~\ref{cohom_class} and~\ref{cohom_class_group} 
give topological indices in index formulas obtained in~\cite{Fds17, BolSa2}.
\end{remark}

The reported study was funded by RFBR and DFG, project number 21-51-12006.

\end{document}